\documentclass[10pt,a4paper,oneside,reqno]{amsart}
\usepackage{setspace}
\usepackage[totalwidth=14.3cm,totalheight=22.6cm]{geometry}
\usepackage[T1]{fontenc}
\usepackage[english]{babel}
\usepackage[autostyle]{csquotes}
   \MakeAutoQuote{‘}{’}
   \MakeOuterQuote{"}
\usepackage{graphicx}
\usepackage[all,cmtip]{xy} 
\SelectTips{eu}{12} 
\usepackage{caption}
\usepackage{epsfig}
\usepackage{amsmath,amssymb,amscd,amsthm}
\usepackage{subfigure}
\usepackage{latexsym}
\usepackage{graphics}
\usepackage{colortbl} 
\usepackage{color}
\usepackage{ae}
\usepackage{enumitem}
\usepackage[all]{xy}
\usepackage{scalerel}
\usepackage{tikz}
\usepackage{cancel}
\usepackage{bbm,amsbsy}
\usepackage{comment}
\usepackage{mathrsfs}  
\usepackage[colorlinks=true,pagebackref=true]{hyperref}
\usepackage[normalem]{ulem}
\usepackage{nicefrac}
\usepackage{xfrac}
\usepackage{faktor}
\usepackage{tikz-cd} 
\usepackage{systeme}
\usepackage{nomencl}
\usepackage[normalem]{ulem}
\makenomenclature

\makeatletter
\newtheorem*{rep@theorem}{\rep@title}
\newcommand{\newreptheorem}[2]{%
\newenvironment{rep#1}[1]{%
 \def\rep@title{#2 \ref{##1}}%
 \begin{rep@theorem}}%
 {\end{rep@theorem}}}
\makeatother

\makeatletter
\newtheorem*{rep@cor}{\rep@title}
\newcommand{\newrepcor}[2]{%
\newenvironment{rep#1}[1]{%
 \def\rep@title{#2 \ref{##1}}%
 \begin{rep@cor}}%
 {\end{rep@cor}}}
\makeatother

\makeatletter
\newtheorem*{rep@prop}{\rep@title}
\newcommand{\newrepprop}[2]{%
\newenvironment{rep#1}[1]{%
 \def\rep@title{#2 \ref{##1}}%
 \begin{rep@prop}}%
 {\end{rep@prop}}}
\makeatother

\newtheorem{cor}{Corollary}[subsection]
\newtheorem{corx}{Corollary}

\newtheorem{theorem}[cor]{Theorem}

\newtheorem{prop}[cor]{Proposition}

 \allowdisplaybreaks
\newrepcor{cor}{Corollary}
\newreptheorem{theorem}{Theorem}
\newrepprop{prop}{Proposition}

\newtheorem{lemma}[cor]{Lemma}

\theoremstyle{definition}
\newtheorem{defi}[cor]{Definition}
\theoremstyle{remark}

\newtheorem{remark}[cor]{Remark}
\newtheorem*{remark*}{Remark}

\newtheorem*{notation*}{Notation}
\theoremstyle{plain}

\newtheorem{principal}[corx]{Theorem}

\newlist{steps}{enumerate}{1}
\setlist[steps, 1]{itemsep=8pt,leftmargin=0cm,itemindent=.5cm,labelwidth=\itemindent,labelsep=0cm,align=left,label = \textbf{\emph{Step \arabic*}:\,}}

\makeatletter
\newcommand{\myitem}[1]{%
\item[#1]\protected@edef\@currentlabel{#1}%
}
\makeatother

\begin{document}\raggedbottom

\setcounter{secnumdepth}{3}
\setcounter{tocdepth}{2}

\title{Harmonic Extension of Weil-Petersson Circle Homeomorphisms}
\thanks{}

\author[Farid Diaf, Vladimir Markovi\'c, Abderrahim Mesbah]{Farid Diaf \and Vladimir Markovi\'c \and Abderrahim Mesbah}

\address{Farid Diaf: Institut de Recherche Mathématique Avancée, UMR 7501, Université de Strasbourg et CNRS, 7 rue René Descartes, 67000 Strasbourg, France.} 
\email{f.diaf@unistra.fr}

\address{Vladimir Markovi\'c: YMSC, Tsinghua University.}

\address{Abderrahim Mesbah:
Beijing Institute of Mathematical Sciences and Applications, Beijing, China.}
\email{abderrahimmesbah@bimsa.cn}

\begin{abstract}
In this paper, we study Weil--Petersson circle homeomorphisms from the viewpoint of harmonic maps. We prove that a homeomorphism $\varphi:\mathbb S^1\to\mathbb S^1$ is Weil--Petersson if and only if its unique quasiconformal harmonic extension to the hyperbolic disk $\mathbb D$ has square-integrable Beltrami differential.

Our approach is based on the anti-holomorphic $L^2$-energy of harmonic maps. We show that this energy is finite for the quasiconformal harmonic extension of every Weil--Petersson circle homeomorphism, and that, among suitable quasiconformal extensions, the harmonic extension minimizes this energy.
\end{abstract}

\maketitle
\tableofcontents

\section{Introduction}
The Weil--Petersson Teichm\"uller space is an important subspace of the universal Teichm\"uller space. In their seminal paper, Takhtajan and Teo \cite{TakhtajonTeo} introduced an infinite-dimensional K\"ahler manifold structure on the universal Teichm\"uller space, which endows it with the structure of a topological group and decomposes it into uncountably many connected components. The Weil--Petersson Teichm\"uller space is precisely the identity component. The significance of the Weil--Petersson Teichm\"uller space extends beyond Teichm\"uller theory to several areas of mathematics, including probability theory \cite{wang-inven,lowner2}, hyperbolic geometry \cite{bishop,renormilized}, operator theory \cite{ShenWeilPeterssonTeichmullerspace,arXiv:lowner-welding}, and anti-de Sitter geometry \cite{optimisation,arXiv:maximalsurface}.

On the other hand, harmonic map theory is a widely studied subject in Teichm\"uller theory and has become an important tool in modern geometry, playing a fundamental role in the proofs of numerous notable results; see, for instance, \cite{ells-sampson,Schoen1993,Corlette,wolf-harmonic,Li-Tam,markovic-schoenconjecture,benoist-hulin-jems}.

The goal of this paper is to characterize the Weil--Petersson Teichm\"uller space in terms of harmonic map theory.

\subsection{Harmonic extension}

The universal Teichm\"uller space is defined as the space of quasisymmetric homeomorphisms of the circle modulo the action of the group of orientation-preserving isometries of the hyperbolic disk $\mathbb D$ by precomposition, where the action on the circle is given by M\"obius transformations. Recall that a homeomorphism $\varphi:\mathbb S^1\to\mathbb S^1$ is called quasisymmetric if it admits a quasiconformal extension $\Phi:\mathbb D\to\mathbb D$ to the hyperbolic disk.

Recall that the Beltrami differential of a quasiconformal map $\Phi:\mathbb{D}\to\mathbb{D}$ is defined by
\[
\mu_{\Phi}:=\frac{\Phi_{\overline z}}{\Phi_z},
\]
where $\Phi_z$ and $\Phi_{\overline z}$ denote the usual complex derivatives with respect to $z$ and $\overline z$, respectively. We will use the following definition of Weil--Petersson circle homeomorphisms.
\begin{defi}\label{defi:WP}
A homeomorphism $\varphi:\mathbb{S}^1\to\mathbb{S}^1$ is called \emph{Weil--Petersson} if it admits a quasiconformal extension $\Phi:\mathbb{D}\to\mathbb{D}$ such that
\[
\int_{\mathbb D}\bigl|\mu_{\Phi}\bigr|^2\,dA_{\mathbb D}<\infty,
\]
where $dA_{\mathbb{D}}=\frac{4\,dx\,dy}{(1-\lvert z\rvert^2)^2}$
denotes the hyperbolic area form on $\mathbb{D}$.
\end{defi}

The construction of \emph{conformally natural} extensions of circle homeomorphisms to the hyperbolic disk is a classical problem in Teichm\"uller theory. Well-known examples include the Douady--Earle extension \cite{DouadyEarle1986}, the harmonic extension \cite{markovic-schoenconjecture}, and the minimal Lagrangian extension \cite{Bonsante-Schlenker}. Given such extensions, a natural line of investigation is to study their Beltrami differentials in relation to the regularity or complexity of the boundary homeomorphism. For Weil--Petersson circle homeomorphisms, Cui proved in \cite{cui2000integrably} that the Beltrami differential of the Douady--Earle extension is square-integrable. More recently, an analogous result for the minimal Lagrangian extension was established in \cite{arXiv:maximalsurface}.

Our main result provides a characterization of Weil--Petersson circle homeomorphisms in terms of harmonic extensions. It was shown by the second author \cite{markovic-schoenconjecture} that every quasisymmetric homeomorphism admits a quasiconformal harmonic extension to the hyperbolic disk, completing a conjecture of Schoen \cite{Schoen1993}. The uniqueness of this extension was previously established by Li--Tam \cite{Li-Tam}.

\begin{principal}\label{mainthm}
A homeomorphism $\varphi:\mathbb{S}^1\to\mathbb{S}^1$ is Weil--Petersson if and only if the Beltrami differential of its unique quasiconformal harmonic extension
$\Phi:\mathbb{D}\to\mathbb{D}$ is square-integrable; that is,
\[
\int_{\mathbb D}\bigl|\mu_{\Phi}\bigr|^2\,dA_{\mathbb D}<\infty.
\]
\end{principal}
Closely related to the Beltrami differential is the \emph{Hopf differential} of a harmonic map. We will show in Proposition~\ref{beltrami-L2-Hopf_L2} that, for a quasiconformal harmonic map, the square-integrability of the Beltrami differential is equivalent to the square-integrability of the Hopf differential in a suitable sense.

We note that in \cite[Theorem~5.8]{harmonicdiffeo}, Tam and Wan established a characterization of smooth circle diffeomorphisms in terms of their harmonic extensions. It is known from \cite{TakhtajonTeo} that smooth diffeomorphisms are Weil--Petersson, and that the Weil--Petersson Teichm\"uller space is precisely the completion of the space of smooth diffeomorphisms with respect to the Weil--Petersson topology.~Our result goes beyond the smooth setting by providing a characterization of the full class of Weil--Petersson circle homeomorphisms.

We also recall that Weil--Petersson homeomorphisms are contained in the class of \emph{symmetric homeomorphisms}. The latter are characterized by the existence of a quasiconformal extension which is \emph{asymptotically conformal}, meaning that its Beltrami differential tends to zero toward the boundary circle. It was shown in \cite[Theorem~2]{Symmetric-harmonic} that the quasiconformal harmonic extension of a symmetric homeomorphism is asymptotically conformal.

\subsection{Anti-holomorphic $L^2$-energy}

We now state our next results, which are of independent interest and whose relation to Theorem~\ref{mainthm} will be explained below. A fundamental property of harmonic maps between closed hyperbolic surfaces is that they minimize the \emph{energy} in their homotopy class. In the universal setting, the energy of a harmonic map $\Phi:\mathbb D\to\mathbb D$ is infinite (see Remark~\ref{remark:on_energy}). Our second main result is therefore to show that the anti-holomorphic $L^2$-energy is finite in the setting of Weil--Petersson circle homeomorphisms, and our third main result establishes an analogous minimality property.

Given a map $\Phi:\mathbb D\to\mathbb D$, we define its \emph{anti-holomorphic $L^2$-energy} by
\[
\int_{\mathbb D} \bigl|\overline{\partial}\Phi(z)\bigr|^2\,dA_{\mathbb D},
\]
where
\[
\overline{\partial}\Phi(z)
=
\frac{1-|\Phi(z)|^2}{1-|z|^2}\,\Phi_{\bar z}(z).
\]
Then we show the following.

\begin{principal}\label{mainprop}
A homeomorphism $\varphi:\mathbb S^1\to\mathbb S^1$ is Weil--Petersson if and only if the anti-holomorphic $L^2$-energy of its unique quasiconformal harmonic extension $\Phi:\mathbb D\to \mathbb D$ is finite, that is,
\[
\int_{\mathbb D}\bigl|\overline{\partial}\Phi\bigr|^2\,dA_{\mathbb D}<\infty.
\]
\end{principal}

For closed hyperbolic surfaces, the harmonic map in a given homotopy class minimizes also the anti-holomorphic $L^2$-energy. The next result shows that, under suitable assumptions, an analogous minimality statement holds in the universal setting.

\begin{principal}\label{main2}
Let $\varphi$ be a Weil--Petersson circle homeomorphism and let $\Phi$ be its harmonic quasiconformal extension. Then, for every $C^\infty$-quasiconformal extension $G:\mathbb D\to\mathbb D$ of $\varphi$ whose Beltrami differential is square-integrable and has finite $C_h^2$-norm, one has
\[
\int_{\mathbb D} |\overline{\partial}\Phi|^2\,dA_{\mathbb D}
\leq
\int_{\mathbb D} |\overline{\partial}G|^2\,dA_{\mathbb D}.
\]
\end{principal}

By a Beltrami differential of finite $C_h^2$-norm, we mean a Beltrami differential whose hyperbolic derivatives up to order two are bounded. For clarity, we omit the exact definition here and refer the reader to Section~\ref{sec3}. The Douady--Earle extension provides one example of a quasiconformal map whose Beltrami differential has finite $C_h^2$-norm; this is proved in Proposition~\ref{prop: beltrami_DE}. There are also many other examples, including quasiconformal harmonic diffeomorphisms and minimal Lagrangian diffeomorphisms; see Remark \ref{remark on C2 bounded}.

\subsection{Outline of the proof}

We now briefly explain the proof of our main results. First, let us explain how Theorem~\ref{mainprop} implies Theorem~\ref{mainthm}. The key observation is that the square-integrability of the Beltrami differential of a harmonic map is equivalent to the finiteness of its anti-holomorphic $L^2$-energy. Indeed, for a quasiconformal harmonic map $\Phi:\mathbb D\to\mathbb D$, it follows from results of Wan \cite{Wan} and the second author \cite{markovic-schoenconjecture} that
\[
C\,|\overline{\partial}\Phi|
\leq
|\mu_\Phi|
\leq
|\overline{\partial}\Phi|,
\]
where $C>0$ depends only on the quasiconformal dilatation of $\Phi$. Therefore, Theorem~\ref{mainthm} reduces to proving that the anti-holomorphic $L^2$-energy of the harmonic quasiconformal extension of a Weil--Petersson circle homeomorphism is finite.

The next observation concerns harmonic maps between closed hyperbolic surfaces. Let $\Sigma$ be a closed hyperbolic surface and let $f:\Sigma\to\Sigma$ be a map. Then the unique harmonic map $h:\Sigma\to\Sigma$ homotopic to $f$ minimizes the anti-holomorphic energy among all maps in the homotopy class of $f$. This follows from a result of Eells--Sampson \cite{ells-sampson} together with Lemma~\ref{lemma:antihol_is_minimum}.

The proof of Theorem~\ref{main2} proceeds by approximating $G$ with a sequence of maps $G_n:\mathbb D\to\mathbb D$ that are equivariant under the action of cocompact Fuchsian groups, and then applying the previous observation on the minimization of anti-holomorphic energy for maps between closed hyperbolic surfaces. This approach requires some PDE results concerning the Beltrami equation, which will be explained in Section~\ref{section-elleptic-regularity}.

The only remaining point in the proof of Theorem~\ref{mainprop} is to construct a quasiconformal extension of a Weil--Petersson circle homeomorphism satisfying the hypotheses of Theorem~\ref{main2}, namely, that its Beltrami differential is square-integrable and has finite $C_h^2$-norm. For this, we take the Douady--Earle extension: its Beltrami differential is square-integrable by a result of \cite{cui2000integrably}, and the finiteness of its $C_h^2$-norm is proved in Proposition~\ref{prop: beltrami_DE}.

\subsection{Acknowledgements}

We would like to thank Andrea Seppi for helpful comments on the first draft of this paper.

\section{Preliminaries}

\subsection{Universal Teichm\"uller theory}

We begin by recalling some basic facts on quasiconformal maps. We refer the reader to \cite{hubbard,Markovic-fletcher-book} for a more detailed exposition.

\begin{defi}\label{defi:quasiconform}
Let $f\colon \Omega\to f(\Omega)\subset \mathbb{C}$ be an orientation-preserving homeomorphism defined on an open subset $\Omega\subset\mathbb{C}$. The map $f$ is said to be \emph{quasiconformal} if its weak partial derivatives belong to $L^1_{\mathrm{loc}}(\Omega)$ and if there exists $\mu\in L^\infty(\Omega)$ with $\|\mu\|_\infty<1$ such that
\begin{equation}\label{eq:beltrami}
f_{\overline z}=\mu\,f_z.
\end{equation}
\end{defi}

A quasiconformal map as in Definition~\ref{defi:quasiconform} is called
\emph{$K$-quasiconformal}, where
\[
K=\frac{1+\|\mu\|_\infty}{1-\|\mu\|_\infty}.
\]

Equation~\eqref{eq:beltrami} is called the \emph{Beltrami equation} with coefficient $\mu$. We will provide additional results concerning this equation later in Section \ref{sec3}.

Since we will mainly be interested in boundary values of quasiconformal self-maps of the unit disk, we recall the notion of quasisymmetry on the unit circle. Given a quadruple $Q=(z_1,z_2,z_3,z_4)$ of distinct points on $\mathbb{S}^1$, we define its cross-ratio by
\[
\operatorname{cr}(z_1,z_2,z_3,z_4)
=
\frac{(z_4-z_1)(z_3-z_2)}
{(z_2-z_1)(z_3-z_4)}.
\]
We call a quadruple $Q=(z_1,z_2,z_3,z_4)$ \emph{symmetric} if
\[
\operatorname{cr}(Q)=-1.
\]

\begin{defi}
An orientation-preserving homeomorphism
$\varphi:\mathbb{S}^1\to\mathbb{S}^1$
is called \emph{$M$-quasisymmetric}, for some $M\geq 1$, if
\[
-M\leq \operatorname{cr}(\varphi(Q))
\leq -\frac{1}{M}
\]
for every symmetric quadruple $Q$. We say that $\varphi$ is \emph{quasisymmetric} if it is $M$-quasisymmetric for some $M\geq 1$.
\end{defi}

It is a classical result that every quasiconformal map $\Phi:\mathbb{D}\to\mathbb{D}$ extends continuously to a quasisymmetric homeomorphism $\varphi:\mathbb{S}^1\to\mathbb{S}^1$, and conversely every quasisymmetric homeomorphism of $\mathbb{S}^1$ admits a quasiconformal extension to $\mathbb{D}$ (see, for instance, the Douady--Earle extension \cite{DouadyEarle1986}).

We denote by $\mathrm{QS}(\mathbb{S}^1)$ the group of orientation-preserving quasisymmetric homeomorphisms of $\mathbb{S}^1$. We also define
\[
\mathrm{Isom}(\mathbb{D})
=
\left\{
A=
\begin{pmatrix}
\alpha & \beta\\
\overline{\beta} & \overline{\alpha}
\end{pmatrix}
:
\alpha,\beta\in\mathbb{C},
\ |\alpha|^2-|\beta|^2=1
\right\}
\big/_{A\sim -A},
\]
the group of orientation-preserving isometries of the hyperbolic disk, which acts on $\mathbb{D}$ and $\mathbb{S}^1$ by M\"obius transformations. The \emph{universal Teichm\"uller space} is defined as
\[
T(1)
:=
\mathrm{Isom}(\mathbb{D})\backslash
\mathrm{QS}(\mathbb{S}^1)
\simeq
\left\{
\varphi\in\mathrm{QS}(\mathbb{S}^1)
\;\middle|\;
\varphi(-1)=-1,\,
\varphi(-i)=-i,\,
\varphi(1)=1
\right\}.
\]
The space $T(1)$ is an infinite-dimensional complex manifold containing all classical Teichm\"uller spaces as complex submanifolds, which motivates the terminology \emph{universal Teichm\"uller space}.

We now briefly recall an equivalent description of the universal Teichm\"uller space. Define the space of Beltrami differentials as the unite ball in $L^{\infty}(\mathbb{D})$, that is
\[
\mathrm{Bel}(\mathbb{D})
:=
\left\{
\mu\in L^\infty(\mathbb{D})
\;\middle|\;
\|\mu\|_\infty<1
\right\}.
\]
Given $\mu\in\mathrm{Bel}(\mathbb{D})$, we extend it to
$\hat{\mathbb C}=\mathbb C\cup\{\infty\}$ by reflection across the unit circle:
\[
\hat{\mu}\!\left(\frac{1}{\overline z}\right)
=
\mu(z)\,
\frac{z^2}{\overline z^2},
\qquad z\in\mathbb D.
\]
By the measurable Riemann mapping theorem, there exists a unique quasiconformal homeomorphism
\[
w_\mu:\hat{\mathbb C}\to\hat{\mathbb C}
\]
satisfying the Beltrami equation
\[
(w_\mu)_{\overline z}
=
\hat{\mu}(w_\mu)_z
\]
and fixing the points $-1$, $-i$, and $1$. The map $w_\mu$ preserves $\mathbb S^1$, and its restriction satisfies
\[
w_\mu|_{\mathbb S^1}\in\mathrm{QS}(\mathbb S^1).
\]
This yields an identification
\[
T(1)\simeq \mathrm{Bel}(\mathbb D)/\sim,
\]
where
\[
\mu\sim\nu
\qquad\Longleftrightarrow\qquad
w_\mu|_{\mathbb S^1}
=
w_\nu|_{\mathbb S^1}.
\]

\begin{remark}\label{remark_on_dependence}
The recent work of El Emam and Sagman \cite{ElEmamSagman2024} studies the dependence of solutions of the Beltrami equation \eqref{eq:beltrami} on the coefficient $\mu$. More precisely, they consider the \emph{principal solution} $f^\mu:\mathbb D\to\mathbb D$, namely the unique quasiconformal map satisfying the Beltrami equation with coefficient $\mu$ and the normalization conditions
\[
f^\mu(0)=0, \qquad f^\mu(1)=1.
\]
In particular, it is shown in \cite[Theorem~A and Remark~2.10]{ElEmamSagman2024} that if $\mu$ is smooth, then $f^\mu$ is smooth. Moreover, if $(\mu_n)$ is a sequence of smooth Beltrami differentials converging to $\mu$ in the $C^\infty$ topology on compact subsets of $\mathbb D$, then $f^{\mu_n}$ converges to $f^\mu$ in the $C^\infty$ topology on compact subsets of $\mathbb D$, and the boundary homeomorphisms $f^{\mu_n}|_{\mathbb S^1}$ converge to $f^\mu|_{\mathbb S^1}$.

Since $f^\mu$ and $w_\mu$ have the same Beltrami differential, there exists an isometry $A^\mu$ of the hyperbolic disk such that $w_\mu = A^\mu \circ f^\mu .$ It follows that $w_\mu$ is smooth. Moreover, observe that $A^\mu$ depends only on the points
\[
f^\mu(-1), \qquad f^\mu(-i), \qquad f^\mu(1),
\]
which are pairwise distinct, since $f^\mu$ is a homeomorphism when restricted to the circle. This implies that if $(\mu_n)$ is a sequence of smooth Beltrami differentials converging to $\mu$ in the $C^\infty$ topology on compact subsets of $\mathbb D$, then $w_{\mu_n}=A^{\mu_n}\circ f^{\mu_n}$ converges to $w_\mu=A^\mu\circ f^\mu$ in the $C^\infty$ topology on compact subsets.
\end{remark}

\subsection{Harmonic maps}

We introduce some notation that will be used throughout the paper and recall several classical results concerning harmonic maps. We consider the density of the hyperbolic metric given by
\[
\rho(z)=\frac{2}{1-|z|^2},
\qquad z\in\mathbb D.
\]

\begin{defi}
A $C^2$ map $\Phi:\mathbb D\to\mathbb D$ is called \emph{harmonic} if and only if it satisfies
\[
\Phi_{z\bar z}
+
2\left(
\frac{\rho_z}{\rho}
\circ\Phi
\right)
\Phi_z\Phi_{\bar z}
=
0.
\]
\end{defi}
For any smooth map \(F:\mathbb D\to\mathbb D\), we define
\[
\partial F(z)
=
\frac{\rho(F(z))}{\rho(z)}\,F_z(z),
\qquad
\overline{\partial}F(z)
=
\frac{\rho(F(z))}{\rho(z)}\,F_{\bar z}(z).
\]
It is straightforward to verify that for every
\(A\in\mathrm{Isom}(\mathbb D)\),
\[
\partial(A\circ F)(z)
=
\frac{A'(F(z))}{|A'(F(z))|}\,
\partial F(z),
\qquad
\overline{\partial}(A\circ F)(z)
=
\frac{A'(F(z))}{|A'(F(z))|}\,
\overline{\partial}F(z).
\]
In particular,
\begin{equation}\label{eq:invariance}
|\partial(A\circ F)|
=
|\partial F|,
\qquad
|\overline{\partial}(A\circ F)|
=
|\overline{\partial}F|.
\end{equation}
When $\Phi:\mathbb D\to\mathbb D$ is harmonic, we refer to
$|\overline{\partial}\Phi|$
and
$|\partial\Phi|$
as the \emph{anti-holomorphic} and \emph{holomorphic energies} of $\Phi$, respectively. Next, we denote by
\[
\mathrm{Hopf}(\Phi)
:=
(\rho^2\circ\Phi)\,
\Phi_z\Phi_{\bar z}
\]
the \emph{Hopf differential} of $\Phi$. It is well known that if $\Phi$ is harmonic, then $\mathrm{Hopf}(\Phi)$ is holomorphic on $\mathbb D$. We recall the following fundamental estimate in the theory of harmonic maps.

\begin{theorem}\cite{Wan,markovic-schoenconjecture}\label{theorem:estimate}
Let $\Phi:\mathbb{D}\to \mathbb{D}$ be a quasiconformal harmonic map. Then there exists a constant $C>0$, depending only on the quasiconformal dilatation of $\Phi$, such that
$$1\leq \lvert \partial\Phi\rvert\leq C.$$
In particular, since
$|\mu_\Phi|=\frac{|\overline{\partial}\Phi|}{|\partial\Phi|},$ we have
\begin{equation}\label{eq:equivalence-bel-and-anti}
\frac{1}{C}\, |\overline{\partial}\Phi|
\leq
|\mu_\Phi|
\leq
|\overline{\partial}\Phi|,
\end{equation}
\end{theorem}

The relation between the Hopf differential and the Beltrami differential is given by
\begin{equation}\label{eq:hopf and beltrami}
|\mu_\Phi|
=
\frac{|\mathrm{Hopf}(\Phi)|}
{\rho^2\,|\partial\Phi|^2}.
\end{equation}
It is possible to characterize quasiconformality of harmonic maps in terms of the Hopf differential. We introduce the space
\[
\mathrm{BQD}(\mathbb{D})
=
\left\{
\Psi:\mathbb{D}\to\mathbb{C}
\text{ holomorphic}
\;\middle|\;
\sup_{z\in\mathbb D}\rho(z)^{-2}|\Psi(z)|
=
\|\Psi\|_{\infty}
<
\infty
\right\},
\]
consisting of holomorphic quadratic differentials with bounded hyperbolic norm. A result of Wan \cite{Wan} states that a harmonic map $\Phi:\mathbb D\to\mathbb D$ is quasiconformal if and only if its Hopf differential belongs to $\mathrm{BQD}(\mathbb D)$.

We also define
\[
A^2(\mathbb D)
=
\left\{
\Psi:\mathbb D\to\mathbb C
\text{ holomorphic}
\;\middle|\;
\int_{\mathbb D}
|\Psi|^2\rho^{-2}\,d^2z
<
\infty
\right\}.
\]

The following lemma gives a sufficient condition for quasiconformality.

\begin{lemma}\label{lemma:L2hopf}
Let $\Phi:\mathbb D\to\mathbb D$ be a harmonic map whose Hopf differential belongs to $A^2(\mathbb D)$. Then $\Phi$ is quasiconformal.
\end{lemma}

\begin{proof}
By \cite[Lemma~I.2.1]{TakhtajonTeo}, we have the inclusion
\[
A^2(\mathbb D)\subset \mathrm{BQD}(\mathbb D).
\]
The conclusion therefore follows from Wan's characterization.
\end{proof}

It would be interesting to investigate whether the square-integrability of the Beltrami differential of a harmonic map implies quasiconformality. Although we do not address this question here, we establish the following relation between the square-integrability of the Beltrami differential and that of the Hopf differential under the quasiconformality assumption.

\begin{prop}\label{beltrami-L2-Hopf_L2}
Let $\Phi:\mathbb D\to\mathbb D$ be a quasiconformal harmonic map. Then the following assertions are equivalent:
\begin{itemize}
\item $\mu_\Phi\in L^2(\mathbb D,dA_{\mathbb D})$;
\item $\mathrm{Hopf}(\Phi)\in A^2(\mathbb D)$.
\end{itemize}
\end{prop}

\begin{proof}
For quasiconformal harmonic maps, Theorem~\ref{theorem:estimate} together with the identity \eqref{eq:hopf and beltrami} implies the equivalence.
\end{proof}

\color{black}

In the proof of our main results, we will also consider harmonic maps between closed hyperbolic surfaces. Let $\Gamma_1,\Gamma_2\leq\mathrm{Isom}(\mathbb D)$ be discrete groups acting freely and cocompactly on $\mathbb D$ so that \begin{equation}\label{eq:closed_hyperbolic_surface}
    S_1=\mathbb D/\Gamma_1,
\qquad
S_2=\mathbb D/\Gamma_2 
\end{equation} are closed hyperbolic surfaces. A map $\Phi:S_1\to S_2$
is said to be harmonic if it lifts to a harmonic map on $\mathbb D$, which we continue to denote by $\Phi$. By the equivariance of the lift and the invariance property \eqref{eq:invariance}, the quantities $|\partial\Phi|$, $ |\overline{\partial}\Phi|$ descend to well-defined functions on $S_1$. The \emph{energy} of $\Phi:S_1\to S_2$ is defined by
\begin{equation}\label{eq:energy_formula}
    E(\Phi)
=
\int_{S_1}
\left(
|\partial\Phi|^2
+
|\overline{\partial}\Phi|^2
\right)
\,dA_{S_1},
\end{equation}
where $dA_{S_1}$ denotes the hyperbolic area form on $S_1$. Note that if $F\subset\mathbb D$ is a fundamental domain for the action of $\Gamma_1$, then
\[
E(\Phi)
=
\int_F
\left(
|\partial\Phi|^2
+
|\overline{\partial}\Phi|^2
\right)
\, dA_{\mathbb{D}}.
\]
Finally, one can show that $\Phi:S_1\to S_2$ is harmonic if and only if it is a critical point of the energy functional (see \cite{ells-sampson,wolf-harmonic} for a more detailed exposition). We conclude this preliminary section on harmonic maps with the following key lemma, which should be well known to experts.

\begin{lemma}\label{lemma:antihol_is_minimum}
Let $S_1,S_2$ be closed hyperbolic surfaces as in \eqref{eq:closed_hyperbolic_surface} and let $\Phi:S_1\to S_2$ be a harmonic diffeomorphism. Let $F:S_1\to S_2$ be an orientation-preserving diffeomorphism homotopic to $\Phi$. Then
\[
\int_{S_1} |\overline{\partial}\Phi|^2\,dA_{S_1}
\leq
\int_{S_1} |\overline{\partial}F|^2\,dA_{S_1},
\]
where $dA_{S_1}$ denotes the hyperbolic area form on $S_1$.
\end{lemma}

\begin{proof}
It follows from a result of Eells--Sampson \cite{ells-sampson} that the harmonic map $\Phi$ minimizes the energy functional in its homotopy class, and so
\[
E(\Phi)\leq E(F).
\]
Next, observe that for any orientation-preserving diffeomorphism $G:S_1\to S_2$, the Jacobian of $G$ is given by
\[
\operatorname{Jac}(G)
=
|\partial G|^2-|\overline{\partial}G|^2.
\]
In particular,
\[
\int_{S_1}\operatorname{Jac}(G)\,dA_{S_1}
=
\operatorname{Area}(S_2).
\]
Applying this identity to both $\Phi$ and $F$, we obtain
\[
\int_{S_1}\operatorname{Jac}(\Phi)\,dA_{S_1}
=
\int_{S_1}\operatorname{Jac}(F)\,dA_{S_1}
=
\operatorname{Area}(S_2).
\]
Using this equality together with formula~\eqref{eq:energy_formula}, we obtain
\[
2\int_{S_1}
|\overline{\partial}\Phi|^2\,dA_{S_1}
=
E(\Phi)-\operatorname{Area}(S_2)
\leq
E(F)-\operatorname{Area}(S_2)
=
2\int_{S_1}
|\overline{\partial}F|^2\,dA_{S_1}.
\]
This completes the proof.
\end{proof}

\begin{remark}\label{remark:on_energy}
The energy of a quasiconformal harmonic map $\Phi:\mathbb{D}\to \mathbb{D}$, in the sense of
$\int_{\mathbb{D}}\left(|\partial \Phi|^2+|\overline{\partial}\Phi|^2\right)\, dA_{\mathbb{D}},$
is infinite. Indeed, by Theorem \ref{theorem:estimate}, we have $|\partial \Phi|\geq 1$, and therefore
\[
|\partial \Phi|^2 + |\overline{\partial}\Phi|^2 \geq 1.
\] and so the integral should be infinite.
\end{remark}

\section{Anti-holomorphic energy}\label{sec3}

In this section, we study in more detail the relation between the holomorphic and anti-holomorphic energies of a harmonic map and its Beltrami coefficient.

\subsection{Elliptic estimate of the Beltrami equation}\label{section-elleptic-regularity}
Let $g_{\mathbb D}$ denote the hyperbolic metric on $\mathbb D$. Viewing Beltrami differentials as $(-1,1)$-tensors, the metric $g_{\mathbb D}$ and its Levi--Civita connection $\nabla$ naturally extend to the bundle of $C^2$ Beltrami differentials. We then define
\begin{equation}\label{eq:norme}
\|\mu\|_{C^2_h(\mathbb D)}
=
\sup_{z\in\mathbb D} |\mu(z)|
+
\sup_{z\in\mathbb D} \lVert\nabla\mu(z)\rVert_{g_{\mathbb D}}
+
\sup_{z\in\mathbb D} \lVert\nabla^2\mu(z)\rVert_{g_{\mathbb D}}.
\end{equation}
The subscript \(h\) in \(C^2_h\) refers to the hyperbolic metric.

The group $\mathrm{Isom}(\mathbb D)$ of orientation-preserving isometries of $\mathbb D$ acts on the space of Beltrami differentials by pullback:
$$
(A^*\mu)(a)
:=
\mu(A(a))
\frac{\overline{A'(a)}}{A'(a)},
$$
for $A\in\mathrm{Isom}(\mathbb D)$ and $a\in\mathbb D$. With respect to this action, each term appearing in the norm \eqref{eq:norme} is invariant. In particular,
\[
\|A^*\mu\|_{C^2_h}
=
\|\mu\|_{C^2_h}.
\]
For every $r>0$, we denote by $\mathbb{D}_r=\{z\in\mathbb{D} : |z|<r\}$. On $\mathbb{D}_r$, the hyperbolic and Euclidean metrics are bilipschitz equivalent, with constants depending only on $r$. The next result is a Schauder-type estimate for the Beltrami equation. It can also be deduced from the more general elliptic estimates recently established by El Emam and Sagman in \cite[Theorem~2.3]{ElEmamSagman2024}. For the convenience of the reader, we provide an alternative proof in our setting.
\begin{prop}\label{prop:Schauder}
Let $\mu\in \mathrm{Bel(\mathbb{D})}$ be a $C^2$ Beltrami differential and let \(F:\mathbb{D}\to \mathbb{D}\) be a $C^2$ solution of the Beltrami equation
\[
F_{\bar z}-\mu F_z=0,
\qquad \|\mu\|_{\infty}<k<1,
\qquad \|\mu\|_{C_h^2}<K <\infty.
\]
Let $0<r<1$, then there exists a constant $C(r,k,K)$ depending only on $r,k$ and $K$ so that 
$$\lVert F\rVert_{C^{1}(\mathbb{D}_{\frac{r}{2}})}<C(k,r,K)$$
\end{prop}
Given a Beltrami differential $\mu$, we denote by $\mathrm{Im} \mu$ its imaginary part. Then we have the following result.

\begin{prop}\cite[Theorem 16.1.6]{Beltrami_PDE}
Let $\mu\in \mathrm{Bel}(\mathbb{D})$ be a $C^2$ Beltrami differential and consider \(F=u+iv\)  a solution of the Beltrami equation
\[
F_{\bar z}-\mu F_z=0,.
\]
Then the real and imaginary parts \(u\) and \(v\) of $F$ satisfy the same second-order elliptic equation
\[
\operatorname{div}(A_\mu \mathrm{grad} (u))=0,
\qquad
\operatorname{div}(A_\mu \mathrm{grad} (v))=0,
\]
where
\[
A_\mu=\frac1{1-|\mu|^2}
\begin{pmatrix}
|1-\mu|^2 & -2\,\mathrm{Im} \mu\\[2mm]
-2\,\mathrm{Im} \mu & |1+\mu|^2
\end{pmatrix}.
\]
\end{prop}

\begin{proof}[Proof of Proposition \ref{prop:Schauder}]
Let \[
A_\mu(z)=\frac{1}{1-|\mu(z)|^2}
\begin{pmatrix}
|1-\mu(z)|^2 & -2\,\mathrm{Im} \mu(z)\\[2mm]
-2\,\mathrm{Im} \mu(z) & |1+\mu(z)|^2
\end{pmatrix}.
\] 
By a direct computation, one can  show that
\(A_\mu(z)\) has eigenvalues
\[
\lambda_1(z)=\frac{1+|\mu(z)|}{1-|\mu(z)|},\qquad
\lambda_2(z)=\frac{1-|\mu(z)|}{1+|\mu(z)|}.
\]
In particular,
\[
\lambda_2(z)\,|\xi|^2 \le \xi^T A_\mu(z)\,\xi \le \lambda_1(z)\,|\xi|^2
\qquad \text{for all } \xi\in\mathbb{R}^2.
\]
Since \(|\mu(z)|\le k\), we obtain
\[
\frac{1-k}{1+k}\le \lambda_2(z)\le \lambda_1(z)\le \frac{1+k}{1-k}.
\]
Therefore,
\[
\frac{1-k}{1+k}\,|\xi|^2
\le \xi^T A_\mu(x)\,\xi
\le \frac{1+k}{1-k}\,|\xi|^2,
\]
which proves that the operator $L=\mathrm{div}(A_\mu\mathrm{grad} u)$ is strictly elliptic. Since the hyperbolic metric and the euclidean metric are bi-Lipschitz on $\mathbb{D}_r$ with a bi-Lipschitz constant depending only on $r$, then it is not difficult to see that the usual  $C^2$ norm of $\mu$ on $\mathbb{D}_r$ (seen as function $\mathbb{D}_r\to \mathbb{C}$) is bounded from above by a constant depending only $r,k$ and $K$. Therefore by classical interior Schauder estimate (see \cite[Theorem 8.32]{PDEBOOK}), we may deduce that 
$$\vert u\rvert_{C^1(\mathbb{D}_{\frac{r}{2}})}\leq C(k,r,K)\vert u\rvert_{C^0(\mathbb{D}_{r})} $$
We apply the same analysis the imaginary part of $F$, and then we can deduce that 

$$\vert F\rvert_{C^1(\mathbb{D}_{\frac{r}{2}})}\leq C(k,r,K)\vert F\rvert_{C^0(\mathbb{D}_{r})} $$ This completes the proof as $F$ takes value in the unit disk.
\end{proof}

We can deduce the following corollary

\begin{cor}\label{cor:bounded hyp derivatives}
    Let $\mu$ and $F$ as in Proposition \ref{prop:Schauder}. Then 
    $$\lvert \partial F(a)\rvert<C(k,r,K),$$ for all $a\in \mathbb{D}$
\end{cor}

\begin{proof}
Let $a\in \mathbb{D}$ and consider $A, B\in \mathrm{Isom}(\mathbb{D})$, and $G=B\circ F\circ A$, so that $A(0)=a$ and $G(0)=0$ then $G$ is a solution of 
$G_{\overline{z}} - \nu G_z = 0$, with
\[
\nu(z)
=
\mu(A(z))\,\frac{\overline{A'(z)}}{A'(z)},
\]
for all $z\in \mathbb{D}$. That is $\nu=A^*\mu$, but the norm of Beltrami is preserved by isometry of $\mathbb{D}$, and hence $  \|\nu\|_{C^{2}_h}=\|\mu\|_{C_h^{2}}< K$ and so we could apply the Proposition \ref{prop:Schauder} to deduce that 

$$\lvert\partial F(a) \rvert=\lvert\partial G(0)\rvert = \lvert G_z(0)\rvert\leq C(r,k,K) \ $$
\end{proof}
\color{black}

\subsection{Approximation result}\label{Section-approximation}

The goal of this section is to prove the approximation results stated in Propositions~\ref{prop: Beltrami_approximation} and~\ref{approximate-quasi-conformal}. These two propositions constitute an important step toward the proof of Theorem~\ref{main2}.

\begin{prop}\label{prop: Beltrami_approximation}
Let $\mu\in\mathrm{Bel}(\mathbb D)$ be a Beltrami differential with finite $C_h^2$-norm, and let $(\Gamma_n)$ be a sequence of cocompact Fuchsian groups whose injectivity radii $\mathrm{Inj}(\Gamma_n)$ tend to $\infty$. Then there exists a sequence of Beltrami differentials $(\mu_n)$ such that:
\begin{itemize}
\item $\mu_n$ is $\Gamma_n$-equivariant;
\item $\lvert \mu_n\rvert \leq\lvert \mu\rvert $
\item $\mu_n$ converges to $\mu$ in the $C^\infty$ topology on compact subsets of $\mathbb D$;
\item there exists a constant $C>0$, depending only on $\|\mu\|_{C_h^2}$, such that
\[
\|\mu_n\|_{C_h^2}\leq C.
\]
\end{itemize}
\end{prop}

\begin{proof}
Let $F_n$ denote a fundamental domain for $\Gamma_n$. Since $\mathrm{Inj}(\Gamma_n)\to\infty$, we may assume that the hyperbolic ball
\[
B_n:=B_h(0,2r_n),
\]
where $r_n\to\infty$, is contained in $F_n$. We also denote by
\[
B_n':=B_h(0,r_n),
\]
so that $B_n'\subset B_n\subset F_n$. Choose a smooth bump function $\phi:\mathbb R\to[0,1]$ satisfying
\[
\phi(t)=1 \quad \text{for } t\leq 1,
\qquad
\phi(t)=0 \quad \text{for } t\geq 2.
\]

\smallskip
\noindent\textbf{Step 1.} \emph{$\mu_n$ inside the fundamental domain.}
Define
\begin{equation}\label{eq:mu_n}
\mu_n(x)
=
\phi\left(\frac{d_h(0,x)}{r_n}\right)\mu(x),
\qquad x\in F_n.
\end{equation}
By construction,
\[
\mu_n=\mu \text{ on } B_n',
\qquad
\mu_n=0 \text{ on } F_n\setminus B_n,
\]
and moreover $|\mu_n|\leq |\mu|$. In particular, $\|\mu_n\|_\infty\leq \|\mu\|_\infty<1.$

\smallskip
\noindent \noindent\textbf{Step 2.} \emph{Extension to the disk.}
Since $\mu_n$ vanishes in a neighborhood of $\partial F_n$, it extends equivariantly to all of $\mathbb D$ by reflections across the sides of $F_n$. This produces a $\Gamma_n$-equivariant Beltrami differential on $\mathbb D$, which we continue to denote by $\mu_n$. Note also that $\lvert \mu_n\rvert \leq \lvert \mu\rvert$; This shows that $\mu_n$ satisfies the first two items of the proposition.

\smallskip
\noindent\textbf{Step 3.} \emph{Uniform $C_h^2$ estimates and convergence.}
Since $\mu_n$ is $\Gamma_n$-equivariant, it suffices to estimate its $C_h^2$-norm on $F_n$. Moreover, since $\mu_n$ vanishes in a neighborhood of $\partial F_n$, it is enough to perform these estimates in the interior of $F_n$, since the extension by reflections agrees with zero near the boundary.

To establish the uniform $C_h^2$ estimate for $\mu_n$, it suffices to show that the function
\[
x\mapsto\phi\left(\frac{d_h(0,x)}{r_n}\right)
\]
has uniformly bounded hyperbolic derivatives up to order two; namely, the function itself, its hyperbolic gradient, and its hyperbolic Hessian are uniformly bounded independently of $n$. This can be verified directly, for instance using polar coordinates on the hyperbolic disk.

Since $\mu$ has finite $C_h^2$-norm, Leibniz's rule implies the existence of a constant $C>0$, depending only on $\|\mu\|_{C_h^2}$ and on the choice of $\phi$, such that
\[
\|\mu_n\|_{C_h^2}\leq C.
\]

Finally, since $\mu_n=\mu$ on $B_n'$ and $r_n\to\infty$, we conclude that $\mu_n\to\mu$ in the $C^\infty$ topology on compact subsets of $\mathbb D$. This completes the proof.
\end{proof}
We can now prove the following approximation result for quasiconformal maps.

\begin{prop}\label{approximate-quasi-conformal}
Let $\mu$, $(\Gamma_n)$, and $(\mu_n)$ be as in Proposition~\ref{prop: Beltrami_approximation}. Let $w_\mu,w_{\mu_n}:\mathbb{D}^2\to\mathbb{D}^2$
be the unique normalized quasiconformal maps fixing $-1$, $-i$, and $1$, whose Beltrami differentials are $\mu$ and $\mu_n$, respectively. Set
$$K=\frac{1+\|\mu\|_\infty}{1-\|\mu\|_\infty},$$
so that $w_\mu$ is $K$-quasiconformal. Then:
\begin{itemize}
\item for each $n$, the map $w_{\mu_n}$ descends to a quasiconformal map between closed hyperbolic surfaces;
\item the maps $w_{\mu_n}$ are $K$-quasiconformal and converge smoothly and uniformly on compact subsets to $w_\mu$;
\item there exists a constant $C>0$, depending only on the $C_h^2$-norm of $\mu$, such that
\[
|\partial w_{\mu_n}|<C.
\]
\end{itemize}
\end{prop}

\begin{proof}
For the first point, it is enough to consider the cocompact Fuchsian group
\[
\Gamma_n'=w_{\mu_n}\Gamma_n w_{\mu_n}^{-1}.
\] Since $\mu_n$ is $\Gamma_n$-equivariant, the map $w_{\mu_n}$ is equivariant with respect to the actions of $\Gamma_n$ and $\Gamma_n'$, and therefore descends to a quasiconformal map between the closed hyperbolic surfaces $\mathbb D^2/\Gamma_n$ and $\mathbb D^2/\Gamma_n'.$

For the second point, by construction of $\mu_n$ in Proposition~\ref{prop: Beltrami_approximation}, we have $|\mu_n|\leq |\mu|$, and therefore each $w_{\mu_n}$ is $K$-quasiconformal. Since $\mu_n$ converges smoothly and uniformly on compact subsets to $\mu$, the convergence $w_{\mu_n}\to w_\mu$
smoothly and uniformly on compact subsets follows from \cite{ElEmamSagman2024}, see Remark \ref{remark_on_dependence}.

Finally, Corollary \ref{cor:bounded hyp derivatives}, together with the fact that the maps $w_{\mu_n}$ are uniformly quasiconformal and the $\mu_n$ are uniformly bounded in the $C_h^2$-norm, implies the existence of a constant $C>0$, depending only on $K$ and $\|\mu\|_{C_h^2}$, such that
\[
|\partial w_{\mu_n}|<C.
\]

This completes the proof.
\end{proof}

\subsection{Proof of Theorem \ref{main2}}
We have now all the tools to prove Theorem \ref{main2}

\begin{proof}[Proof of Theorem \ref{main2}]

Let $\varphi:\mathbb{S}^1\to \mathbb{S}^1$ be a Weil--Petersson circle homeomorphism (see Definition~\ref{defi:WP}). Up to composing with isometries of $\mathbb{D}^2$, we may assume that $\varphi$ fixes $-1$, $-i$, and $1$. Since $\varphi$ is, in particular, quasisymmetric, it follows from \cite{Li-Tam,markovic-schoenconjecture} that there exists a unique quasiconformal harmonic extension
\[
\Phi:\mathbb{D}^2\to\mathbb{D}^2
\]
of $\varphi$. Consider now another quasiconformal extension $G:\mathbb{D}^2\to\mathbb{D}^2$ of $\varphi$ whose Beltrami differential $\mu_G:=\mu$ has finite $C_h^2$-norm.

Let $\{\Gamma_n\}$ be a sequence of cocompact Fuchsian groups whose injectivity radii tend to infinity, and apply Proposition~\ref{prop: Beltrami_approximation} to obtain Beltrami differentials $\mu_n$ as in that proposition.
Let
$w_{\mu_n},\, w_\mu:\mathbb{D}^2\to\mathbb{D}^2$
be the unique quasiconformal maps fixing $-1$, $-i$, and $1$, with Beltrami differentials $\mu_n$ and $\mu$, respectively. In particular, $G=w_\mu$, and $w_{\mu_n}$ is equivariant with respect to the actions of the Fuchsian groups $\Gamma_n$ and $w_{\mu_n}\Gamma_n w_{\mu_n}^{-1}$, respectively. Moreover, there is a constant $C$, depending only on $\lVert \mu \rVert_{C^{2}_h}$, such that
$$\lvert \partial w_{\mu_n}\rvert \leq C .$$ First, we claim that

\begin{equation}\label{firstclaim}
\int_{\mathbb{D}} |\mu_\Phi|^2 dA_{\mathbb D}
\leq
C
\int_{\mathbb{D}} |\mu|^2 dA_{\mathbb D}.
\end{equation}
Let $F_n$ be a fundamental domain for $\Gamma_n$. Then, by Lemma~\ref{lemma:antihol_is_minimum}, we have
\begin{equation}\label{eq:sequence_of_anti_holomrophic_energy}
\int_{F_n} |\overline{\partial}\Phi_n|^2\,dA_{\mathbb{D}}
\leq
\int_{F_n} |\overline{\partial}w_{\mu_n}|^2\,dA_{\mathbb{D}},
\end{equation}
Using Proposition \ref{prop: Beltrami_approximation}, we have that $|\mu_n|\leq |\mu|$ and so
\begin{equation}\label{eq:dominaition_of_d_bar}
|\overline{\partial}w_{\mu_n}|^2
=
|\mu_n|^2 |\partial w_{\mu_n}|^2
\leq
C^2 |\mu_n|^2\leq C^2\lvert \mu\rvert^2.
\end{equation}
Using Theorem \ref{theorem:estimate}, we get
\[
\int_{F_n} |\mu_{\Phi_n}|^2\, dA_{\mathbb D}
\leq
\int_{F_n} |\overline{\partial}\Phi_n|^2\, dA_{\mathbb D}
\leq
\int_{F_n} |\overline{\partial}w_{\mu_n}|^2\, dA_{\mathbb D}\leq
C^2
\int_{\mathbb{D}} |\mu|^2\, dA_{\mathbb D}.
\]

If we consider now a compact $K\subset \mathbb{D}^2$, then since the injectivity radii converges to infinity, then $K\subset F_n$ for sufficiently big $n$. Now since $w_\mu$ converge smoothly and uniformly on compact
subsets to $w_\mu$ (see Proposition \ref{approximate-quasi-conformal})
and \(\Phi_n\) converges to \(\Phi\) in the \(C^2\) topology uniformly on compact subsets (see \cite[Corollary 7.7 ]{Benoist_Hulin}) then it follows from the dominated convergence theorem, that 
\begin{equation}\label{eq:analysis}
    \int_K |\mu_{\Phi}|^2\, dA_{\mathbb D}\leq C^2\int_{\mathbb{D}} |\mu|^2\, dA_{\mathbb D}.
\end{equation} Since $K$ is taken arbitrary, then the proof of the Claim \eqref{firstclaim} follows. 

To conclude the proof, observe that using \eqref{eq:dominaition_of_d_bar} and the convergence of $\overline{\partial }w_{\mu_n}$ to $\overline{\partial}w_{\mu}$, we can deduce that $\int_{\mathbb{D}} |\overline{\partial}w_\mu|^2\, dA_{\mathbb D}$ is finite and by the same kind of the analysis in \eqref{eq:analysis}, we can pass to the limit in \eqref{eq:sequence_of_anti_holomrophic_energy} to conclude the proof of the Proposition.
\end{proof}

\section{Harmonic extension of Weil--Petersson homeomorphisms}\label{section-douady-earl-derevatives}

The goal of this section is to complete the proof of Theorem~\ref{mainthm}. The main step is to construct a quasiconformal extension of a Weil--Petersson circle homeomorphism satisfying the hypotheses of Theorem~\ref{main2}. This extension will be given by the so-called \emph{Douady--Earle extension}.

\subsection{Douady--Earle extension}

We briefly recall the construction of the Douady--Earle extension. Given a quasisymmetric homeomorphism $\varphi:\mathbb{S}^1\to \mathbb{S}^1$, Douady and Earle constructed in \cite{DouadyEarle1986} a \emph{conformally natural} quasiconformal extension $\mathrm{DE}(\varphi)$ of $\varphi$. By conformal naturality, we mean that for any orientation-preserving isometries $A$ and $B$ of $\mathbb{D}$,
$$
\mathrm{DE}(A\circ \varphi\circ B)=A\circ \mathrm{DE}(\varphi)\circ B.
$$
For any $z\in\mathbb{D}$, the point $\mathrm{DE}(\varphi)(z)\in\mathbb{D}$ is defined as the unique point satisfying
$$
L_\varphi(z,\mathrm{DE}(\varphi)(z))=0,
$$
where
$$
L_\varphi(z,w)=
\frac{1}{2\pi}
\int_{\mathbb{S}^1}
\frac{\varphi(\xi)-w}{1-\overline{w}\varphi(\xi)}
\frac{1-|z|^2}{|z-\xi|^2}
\,|d\xi|.
$$
The map $\mathrm{DE}(\varphi):\mathbb{D}\to\mathbb{D}$ is a real analytic quasiconformal diffeomorphism extending $\varphi$. Moreover, viewed as a map from $\overline{\mathbb D}$ to itself, it is continuous.

We endow the space of quasisymmetric homeomorphisms $\mathrm{QS}(\mathbb S^1)$ with the $C^0$ topology and the space $C^\infty(\mathbb D)\cap C^0(\overline{\mathbb D})$ with the topology induced by the diagonal embedding into
$$
C^\infty(\mathbb D)\times C^0(\overline{\mathbb D}),
$$
where convergence in $C^\infty(\mathbb D)$ (resp. $C^0(\overline{\mathbb D})$) means convergence in the $C^\infty$ (resp. $C^0$) topology on compact subsets. Then we have the following result.

\begin{prop}\cite[Proposition~2]{DouadyEarle1986}\label{prop:smooth-dependence-of-douady-earl}
The Douady--Earle operator
$$
\begin{array}{ccccc}
\mathrm{DE}
&:&
\mathrm{QS}(\mathbb{S}^1)
&\to&
C^\infty(\mathbb D)\cap C^0(\overline{\mathbb D})
\\
&&
\varphi
&\mapsto&
\mathrm{DE}(\varphi)
\end{array}
$$
is continuous.
\end{prop}

The next result shows that the Beltrami differential of a Douady--Earle extension has finite $C_h^2$-norm. In fact, we prove the following stronger statement.

\begin{prop}\label{prop: beltrami_DE}
Let $M\geq1$. Then there exists a constant $C(M)>0$, depending only on $M$, such that if $\varphi:\mathbb S^1\to\mathbb S^1$ is an $M$-quasisymmetric homeomorphism, then
$$
\|\mu_{\mathrm{DE}(\varphi)}\|_{C_h^2}\leq C(M).
$$
\end{prop}

\begin{proof}

By definition of the $C_h^2$-norm, it suffices to bound each of the quantities
$$
|\mu(z)|,
\qquad
\lVert\nabla\mu(z)\rVert_{g_{\mathbb D}},
\qquad
\lVert\nabla^2\mu(z)\rVert_{g_{\mathbb D}}
$$
by a constant depending only on $M$. We prove this for $\lVert\nabla\mu\rVert_{g_{\mathbb D}}$; the argument for the other terms is identical.

Assume by contradiction that there exists a sequence of $M$-quasisymmetric homeomorphisms $\varphi_n:\mathbb S^1\to\mathbb S^1$ and points $a_n\in\mathbb D$ such that
$$
\lVert\nabla\mu_{\mathrm{DE}(\varphi_n)}(a_n)\rVert_{g_{\mathbb D}}
\to\infty.
$$
Choose isometries $A_n,B_n\in\mathrm{Isom}(\mathbb D)$ such that
$$
A_n(0)=a_n,
\qquad
B_n(\mathrm{DE}(\varphi_n)(a_n))=0.
$$
Define
$\psi_n:=B_n\circ\varphi_n\circ A_n.$ Then $\mathrm{DE}(\psi_n)$ fixes the origin and, by conformal naturality,
$$
\mathrm{DE}(\psi_n)
=
B_n\circ\mathrm{DE}(\varphi_n)\circ A_n.
$$
Therefore,
$$
\mu_{\mathrm{DE}(\psi_n)}
=
A_n^*\mu_{\mathrm{DE}(\varphi_n)},
$$
and by equivariance,
\begin{equation}\label{eq:douadyproof}
\lVert\nabla\mu_{\mathrm{DE}(\psi_n)}(0)\rVert_{g_{\mathbb D}}
=
\lVert\nabla\mu_{\mathrm{DE}(\varphi_n)}(a_n)\rVert_{g_{\mathbb D}}
\to\infty.
\end{equation}

Since the $\psi_n$ are $M$-quasisymmetric, the maps $\mathrm{DE}(\psi_n)$ are uniformly $K$-quasiconformal, where $K$ depends only on $M$. By compactness of  $K$-quasiconformal maps fixing the origin, after passing to a subsequence we may assume that $\mathrm{DE}(\psi_n)$ converges uniformly on compact subsets to a $K$-quasiconformal map. This implies that $\psi_n$ converges to a quasisymmetric homeomorphism $\psi:\mathbb S^1\to\mathbb S^1$ (see for example \cite[Lemma 4.8]{bdms}). Therefore Proposition \ref{prop:smooth-dependence-of-douady-earl} implies that $\mathrm{DE}(\psi_n)$ converges to $\mathrm{DE}(\psi)$ in the $C^\infty$ topology on compact subsets of $\mathbb D$, contradicting \eqref{eq:douadyproof}. This completes the proof.
\end{proof}

\begin{remark}\label{remark on C2 bounded}
We note that an analogous statement to Proposition~\ref{prop: beltrami_DE} holds for quasiconformal harmonic diffeomorphisms and minimal Lagrangian diffeomorphisms. Indeed, both extensions are conformally natural and satisfy an analogue of Proposition~\ref{prop:smooth-dependence-of-douady-earl}; see \cite[Section 7]{Benoist_Hulin} for harmonic maps and \cite[Section 5]{Bonsante-Schlenker} for minimal Lagrangian maps.
\end{remark}

\subsection{Proofs of Theorems \ref{mainprop} and \ref{mainthm}}

We now prove Theorem \ref{mainprop} and Theorem \ref{mainthm}. Let $\varphi:\mathbb{S}^1 \to \mathbb{S}^1$ be a quasisymmetric homeomorphism, and let $\Phi:\mathbb{D}\to\mathbb{D}$
denote the unique quasiconformal harmonic extension of $\varphi$. 

\begin{proof}[Proof of Theorem \ref{mainprop}]
Assume that $\varphi$ is a Weil--Petersson homeomorphism. We show that the anti-holomorphic energy of $\Phi$ is finite. Let $\mathrm{DE}(\varphi):\mathbb{D}\to\mathbb{D}$ denote the Douady--Earle extension of $\varphi$. By \cite{cui2000integrably}, the Beltrami differential of $\mathrm{DE}(\varphi)$ is square-integrable. Moreover, Proposition \ref{prop: beltrami_DE} implies that its $C_h^2$ norm is finite. Therefore, we may apply Theorem \ref{main2} to conclude that the anti-holomorphic energy of $\Phi$ is finite.

Conversely, assume that the anti-holomorphic energy of $\Phi$ is finite. Then, by \eqref{eq:equivalence-bel-and-anti}, the Beltrami differential $\mu_\Phi$ is square-integrable. Hence, by Definition \ref{defi:WP}, the boundary map $\varphi$ is Weil--Petersson. This completes the proof.
\end{proof}

\begin{proof}[Proof of Theorem \ref{mainthm}]
The theorem follows immediately from Theorem \ref{mainprop} together with the estimate \eqref{eq:equivalence-bel-and-anti}.
\end{proof}

\bibliographystyle{alpha}
\bibliography{bib}

@book{TakhtajonTeo,
 author = {Takhtajan, Leon A. and Teo, Lee-Peng},
 title = {Weil-{Petersson} metric on the universal {Teichm{\"u}ller} space},
 fseries = {Memoirs of the American Mathematical Society},
 series = {Mem. Am. Math. Soc.},
 issn = {0065-9266},
 volume = {861},
 isbn = {978-0-8218-3936-2; 978-1-4704-0465-9},
 year = {2006},
 publisher = {Providence, RI: American Mathematical Society (AMS)},
 language = {English},
 doi = {10.1090/memo/0861},
 zbMATH = {5063220},
 Zbl = {1243.32010}
}

@article{harmonicdiffeo,
 author = {Tam, Luen-Fai and Wan, Tom Y. H.},
 title = {Quasi-conformal harmonic diffeomorphism and the universal {Teichm{\"u}ller} space},
 fjournal = {Journal of Differential Geometry},
 journal = {J. Differ. Geom.},
 issn = {0022-040X},
 volume = {42},
 number = {2},
 pages = {368--410},
 year = {1995},
 language = {English},
 doi = {10.4310/jdg/1214457235},
 zbMATH = {894386},
 Zbl = {0873.32019}
}

@article{Wan,
 author = {Wan, Tom Y. H.},
 title = {Constant mean curvature surface, harmonic maps, and universal {Teichm{\"u}ller} space},
 fjournal = {Journal of Differential Geometry},
 journal = {J. Differ. Geom.},
 issn = {0022-040X},
 volume = {35},
 number = {3},
 pages = {643--657},
 year = {1992},
 language = {English},
 doi = {10.4310/jdg/1214448260},
 zbMATH = {120216},
 Zbl = {0808.53056}
}

@article{ShenWeilPeterssonTeichmullerspace,
  title = {Weil-{{Petersson Teichmüller}} Space},
  author = {Shen, Yuliang},
  year = 2018,
  journal = {American Journal of Mathematics},
  volume = {140},
  number = {4},
  pages = {1041--1074},
  issn = {0002-9327,1080-6377},
  doi = {10.1353/ajm.2018.0023},
  mrnumber = {3828040}
}

@article{ElEmamSagman2024,
 author = {El Emam, Christian and Sagman, Nathaniel},
 title = {Holomorphic dependence for the {Beltrami} equation in {Sobolev} spaces},
 fjournal = {Proceedings of the American Mathematical Society},
 journal = {Proc. Am. Math. Soc.},
 issn = {0002-9939},
 volume = {154},
 number = {5},
 pages = {1973--1989},
 year = {2026},
 language = {English},
 doi = {10.1090/proc/17402},
 zbMATH = {8188376}
}

@article{cui2000integrably,
  title={Integrably asymptotic affine homeomorphisms of the circle and Teichm{\"u}ller spaces},
  author={Cui, Guizhen},
  journal={Science in China Series A: Mathematics},
  volume={43},
  number={3},
  pages={267--279},
  year={2000},
  publisher={Springer}
}

@book{Beltrami_PDE,
 author = {Astala, Kari and Iwaniec, Tadeusz and Martin, Gaven},
 title = {Elliptic partial differential equations and quasiconformal mappings in the plane},
 fseries = {Princeton Mathematical Series},
 series = {Princeton Math. Ser.},
 volume = {48},
 isbn = {978-0-691-13777-3},
 year = {2009},
 publisher = {Princeton, NJ: Princeton University Press},
 language = {English},
 zbMATH = {5495288},
 Zbl = {1182.30001}
}

@book{PDEBOOK,
 author = {Gilbarg, David and Trudinger, Neil S.},
 title = {Elliptic partial differential equations of second order},
 edition = {Reprint of the 1998 ed.},
 fseries = {Classics in Mathematics},
 series = {Class. Math.},
 issn = {1431-0821},
 isbn = {3-540-41160-7},
 year = {2001},
 publisher = {Berlin: Springer},
 language = {English},
 zbMATH = {1554166},
 Zbl = {1042.35002}
}

@article{Benoist_Hulin,
 author = {Benoist, Yves and Hulin, Dominique},
 title = {Harmonic quasi-isometries of pinched {Hadamard} surfaces are injective},
 fjournal = {Tunisian Journal of Mathematics},
 journal = {Tunis. J. Math.},
 issn = {2576-7658},
 volume = {4},
 number = {2},
 pages = {307--328},
 year = {2022},
 language = {English},
 doi = {10.2140/tunis.2022.4.307},
 zbMATH = {7584410},
 Zbl = {1503.53129}
}

@article{DouadyEarle1986,
  author = {Douady, Adrien and Earle, Clifford J.},
  title = {Conformally Natural Extension of Homeomorphisms of the Circle},
  journal = {Acta Mathematica},
  volume = {157},
  year = {1986},
  pages = {23--48}
}

@article{Bonsante-Schlenker,
 author = {Bonsante, Francesco and Schlenker, Jean-Marc},
 title = {Maximal surfaces and the universal {Teichm{\"u}ller} space},
 fjournal = {Inventiones Mathematicae},
 journal = {Invent. Math.},
 issn = {0020-9910},
 volume = {182},
 number = {2},
 pages = {279--333},
 year = {2010},
 language = {English},
 doi = {10.1007/s00222-010-0263-x},
 zbMATH = {5818340},
 Zbl = {1222.53063}
}

@article{Li-Tam,
 author = {Li, Peter and Tam, Luen-Fai},
 title = {Uniqueness and regularity of proper harmonic maps},
 fjournal = {Annals of Mathematics. Second Series},
 journal = {Ann. Math. (2)},
 issn = {0003-486X},
 volume = {137},
 number = {1},
 pages = {167--201},
 year = {1993},
 language = {English},
 doi = {10.2307/2946622},
 zbMATH = {221042},
 Zbl = {0776.58010}
}

@article{markovic-schoenconjecture,
 author = {Markovic, Vladimir},
 title = {Harmonic maps and the {Schoen} conjecture},
 fjournal = {Journal of the American Mathematical Society},
 journal = {J. Am. Math. Soc.},
 issn = {0894-0347},
 volume = {30},
 number = {3},
 pages = {799--817},
 year = {2017},
 language = {English},
 doi = {10.1090/jams/881},
 url = {resolver.caltech.edu/CaltechAUTHORS:20170505-094324508},
 zbMATH = {6699459},
 Zbl = {1371.53059}
}

@incollection{Schoen1993,
  author    = {Richard Schoen},
  title     = {The role of harmonic mappings in rigidity and deformation problems},
  booktitle = {Complex Geometry (Osaka, 1990)},
  series    = {Lecture Notes in Pure and Applied Mathematics},
  volume    = {143},
  pages     = {179--200},
  publisher = {Dekker},
  address   = {New York},
  year      = {1993}
}

@article{Corlette,
 author = {Corlette, Kevin},
 title = {Flat {G}-bundles with canonical metrics},
 fjournal = {Journal of Differential Geometry},
 journal = {J. Differ. Geom.},
 issn = {0022-040X},
 volume = {28},
 number = {3},
 pages = {361--382},
 year = {1988},
 language = {English},
 doi = {10.4310/jdg/1214442469},
 zbMATH = {4107768},
 Zbl = {0676.58007}
}

@article{wolf-harmonic,
 author = {Wolf, Michael},
 title = {The {Teichm{\"u}ller} theory of harmonic maps},
 fjournal = {Journal of Differential Geometry},
 journal = {J. Differ. Geom.},
 issn = {0022-040X},
 volume = {29},
 number = {2},
 pages = {449--479},
 year = {1989},
 language = {English},
 doi = {10.4310/jdg/1214442885},
 zbMATH = {4069867},
 Zbl = {0655.58009}
}

@article{ells-sampson,
 author = {Eells, James jun. and Sampson, J. H.},
 title = {Harmonic mappings of {Riemannian} manifolds},
 fjournal = {American Journal of Mathematics},
 journal = {Am. J. Math.},
 issn = {0002-9327},
 volume = {86},
 pages = {109--160},
 year = {1964},
 language = {English},
 doi = {10.2307/2373037},
 zbMATH = {3198580},
 Zbl = {0122.40102}
}

@article{bishop,
 author = {Bishop, Christopher J.},
 title = {Weil-{Petersson} curves, {{\(\beta\)}}-numbers, and minimal surfaces},
 fjournal = {Annals of Mathematics. Second Series},
 journal = {Ann. Math. (2)},
 issn = {0003-486X},
 volume = {202},
 number = {1},
 pages = {111--188},
 year = {2025},
 language = {English},
 url = {annals.math.princeton.edu/wp-content/uploads/annals-v202-n1-p02-p.pdf},
 zbMATH = {8063317}
}

@article{renormilized,
 author = {Bridgeman, Martin and Bromberg, Kenneth and Pallete, Franco Vargas and Wang, Yilin},
 title = {Universal {Liouville} action as a renormalized volume and its gradient flow},
 fjournal = {Duke Mathematical Journal},
 journal = {Duke Math. J.},
 issn = {0012-7094},
 volume = {174},
 number = {13},
 pages = {2821--2876},
 year = {2025},
 language = {English},
 doi = {10.1215/00127094-2025-0007},
 zbMATH = {8109701}
}

@misc{arXiv:lowner-welding,
 author = {Shuo Fan and Fredrik Viklund and Yilin Wang},
 title = {On the {Loewner} energy of a welding homeomorphism},
 year = {2026},
 howpublished = {Preprint, {arXiv}:2604.16737 [math.{CV}] (2026)},
 url = {https://arxiv.org/abs/2604.16737},
 arXiv = {arXiv:2604.16737}
}

@misc{arXiv:maximalsurface,
 author = {Farid Diaf and Alex Moriani and Rym Sma{\"{\i}} and Graham Andrew Smith and Enrico Trebeschi},
 title = {Weil--{Petersson} homeomorphisms, minimal lagrangian diffeomorphisms, and maximal surfaces in anti-de {Sitter} space},
 year = {2026},
 howpublished = {Preprint, {arXiv}:2604.17804 [math.{DG}] (2026)},
 url = {https://arxiv.org/abs/2604.17804},
 arXiv = {arXiv:2604.17804}
}

@article{optimisation,
 author = {Wang, Yilin},
 title = {Two optimization problems for the {Loewner} energy},
 fjournal = {Journal of Mathematical Physics},
 journal = {J. Math. Phys.},
 issn = {0022-2488},
 volume = {66},
 number = {2},
 pages = {13},
 note = {Id/No 023502},
 year = {2025},
 language = {English},
 doi = {10.1063/5.0241113},
 zbMATH = {7993804},
 Zbl = {1559.49027}
}

@article{wang-inven,
 author = {Wang, Yilin},
 title = {Equivalent descriptions of the {Loewner} energy},
 fjournal = {Inventiones Mathematicae},
 journal = {Invent. Math.},
 issn = {0020-9910},
 volume = {218},
 number = {2},
 pages = {573--621},
 year = {2019},
 language = {English},
 doi = {10.1007/s00222-019-00887-0},
 zbMATH = {7114403},
 Zbl = {1435.30074}
}

@article{lowner2,
 author = {Viklund, Fredrik and Wang, Yilin},
 title = {Interplay between {Loewner} and {Dirichlet} energies via conformal welding and flow-lines},
 fjournal = {Geometric and Functional Analysis. GAFA},
 journal = {Geom. Funct. Anal.},
 issn = {1016-443X},
 volume = {30},
 number = {1},
 pages = {289--321},
 year = {2020},
 language = {English},
 doi = {10.1007/s00039-020-00521-9},
 zbMATH = {7184230},
 Zbl = {1436.30009}
}

@book{hubbard,
 author = {Hubbard, John Hamal},
 title = {Teichm{\"u}ller theory and applications to geometry, topology, and dynamics. {Volume} 1: {Teichm{\"u}ller} theory. {With} contributions by {Adrien} {Douady}, {William} {Dunbar}, and {Roland} {Roeder}, {Sylvain} {Bonnot}, {David} {Brown}, {Allen} {Hatcher}, {Chris} {Hruska}, {Sudeb} {Mitra}},
 isbn = {0-9715766-2-9},
 year = {2006},
 publisher = {Ithaca, NY: Matrix Editions},
 language = {English},
 zbMATH = {5042912},
 Zbl = {1102.30001}
}

@article{bdms,
 author = {Bonsante, Francesco and Danciger, Jeffrey and Maloni, Sara and Schlenker, Jean-Marc},
 title = {The induced metric on the boundary of the convex hull of a quasicircle in hyperbolic and anti-de {Sitter} geometry},
 fjournal = {Geometry \& Topology},
 journal = {Geom. Topol.},
 issn = {1465-3060},
 volume = {25},
 number = {6},
 pages = {2827--2911},
 year = {2021},
 language = {English},
 doi = {10.2140/gt.2021.25.2827},
 zbMATH = {7442194},
 Zbl = {1483.30084}
}

@book{Markovic-fletcher-book,
 author = {Fletcher, Alastair and Markovic, Vladimir},
 title = {Quasiconformal maps and {Teichm{\"u}ller} theory},
 fseries = {Oxford Graduate Texts in Mathematics},
 series = {Oxf. Grad. Texts Math.},
 volume = {11},
 isbn = {0-19-856926-2},
 year = {2007},
 publisher = {Oxford: Oxford University Press},
 language = {English},
 zbMATH = {5076257},
 Zbl = {1103.30002}
}

@article{benoist-hulin-jems,
 author = {Benoist, Yves and Hulin, Dominique},
 title = {Harmonic quasi-isometric maps. {II}: {Negatively} curved manifolds},
 fjournal = {Journal of the European Mathematical Society (JEMS)},
 journal = {J. Eur. Math. Soc. (JEMS)},
 issn = {1435-9855},
 volume = {23},
 number = {9},
 pages = {2861--2911},
 year = {2021},
 language = {English},
 doi = {10.4171/JEMS/1065},
 zbMATH = {7367680},
 Zbl = {1482.53082}
}

@article{Symmetric-harmonic,
 author = {Yao, Guowu},
 title = {Convergence of harmonic maps on the {Poincar{\'e}} disk},
 fjournal = {Proceedings of the American Mathematical Society},
 journal = {Proc. Am. Math. Soc.},
 issn = {0002-9939},
 volume = {132},
 number = {8},
 pages = {2483--2493},
 year = {2004},
 language = {English},
 doi = {10.1090/S0002-9939-04-07465-9},
 zbMATH = {2091058},
 Zbl = {1127.30009}
}

\end{document}